\RequirePackage{fix-cm}
\documentclass[smallextended]{svjour3}       
\smartqed  
\usepackage{graphicx}
\usepackage{latexsym,amssymb,amsmath,enumerate,bbm,graphicx,psfrag,color}
\usepackage{stmaryrd} 
\usepackage[latin1]{inputenc}
\usepackage{cite}

\newcommand{\N}{\mathbb{N}}
\newcommand{\R}{\mathbb{R}}

\newcommand{\LL}{\mathcal{L}}

\newcommand{\sS}{\mathbb{S}}
\newcommand{\1}{\ensuremath{\mathbbm{1}}}

\newcommand{\dx}[1][x]{\ensuremath{\,{\rm{d}} #1}}

\def\norm#1{\hspace{0.2ex} \|#1\| \hspace{0.2ex}} 
 
\newcommand\mycom[2]{\genfrac{}{}{0pt}{}{#1}{#2}} 
\newcommand{\kommentar}[1]{}

\def\makeheadbox{\relax}
 
\begin{document}

\title{Solving an inverse elliptic coefficient problem by convex non-linear semidefinite programming
}
\titlerunning{Solving an inverse coefficient problem by convex semidefinite programming}        

\author{Bastian Harrach}


\institute{B. Harrach \at 
              Institute for Mathematics, Goethe-University Frankfurt, Frankfurt am Main, Germany \\
              \email{harrach@math.uni-frankfurt.de} 
}

\date{}

\maketitle

\begin{abstract}
Several applications in medical imaging and non-destructive material testing lead to inverse elliptic coefficient problems,
where an unknown coefficient function in an elliptic PDE is to be determined from partial knowledge of its solutions.
This is usually a highly non-linear ill-posed inverse problem, for which unique reconstructability results, stability estimates and 
global convergence of numerical methods are very hard to achieve.

The aim of this note is to point out a new connection between inverse coefficient problems and semidefinite programming that may help addressing these challenges.
We show that an inverse elliptic Robin transmission problem with finitely many measurements can be equivalently rewritten as 
a uniquely solvable convex non-linear semidefinite optimization problem.
This allows to explicitly estimate the number of measurements that is required to achieve a desired resolution, to derive an error estimate for noisy data, 
and to overcome the problem of local minima that usually appears in optimization-based
approaches for inverse coefficient problems.
\end{abstract}

\section{Introduction}

Inverse elliptic coefficient problems arise in a number of applications in medical imaging and non-destructive material testing.
The arguably most prominent example is the Calder\'on problem \cite{calderon1980inverse,calderon2006inverse} which models electrical impedance tomography (EIT) where 
the electrical conductivity distribution inside a patient is to be determined from current/voltage measurements on its surface, cf.\ \cite{adler2015electrical} for an overview.
Theoretical uniqueness questions for inverse elliptic coefficient problems have mostly been studied in the idealized infinite-dimensional setting where
(intuitively speaking) the unknown coefficient function is to be determined with infinite resolution from infinitely many measurements, cf., e.g., \cite{uhlmann2009electrical,harrach2009uniqueness,kenig2014recent}.
Lipschitz stability results have been obtained for finitely many unknowns and infinitely many measurements in, e.g., \cite{alessandrini2005lipschitz,beretta2015stable,harrach2019global}.
Recently there has been progress on the practically very relevant case of finitely many unknowns and measurements,
cf., e.g., \cite{alberti2019calderon,harrach2019uniqueness,ruland2019lipschitz}. But little is known yet about explicitly characterizing the required number of measurements for a given desired resolution.  

Practical reconstruction algorithms for inverse coefficient problems are usually based on regularized data-fitting, which formulates the
inverse problem as a minimization problem for a residuum functional together with a regularization term. As the residuum formulation is typically non-convex, 
this approach highly suffers from the problem of local minima. 
Convexification approaches for inverse coefficient problems have been studied in, e.g., \cite{klibanov2019convexification}. But, to the knowledge of the author, no equivalent convex reformulations of 
inverse coefficient problems with finitely many measurements have been found yet.

The aim of this work is to show that a uniquely solvable convex reformulation of an inverse coefficient problem is indeed possible if enough measurements are being taken, and that the required number of measurements can be explicitly characterized. More precisely, we state a criterion that is sufficient for unique solvability and 
for the solution minimizing a linear cost functional under a convex non-linear semidefinite constraint. For a given desired resolution and a given number of measurements, the criterion can be explicitly checked by calculating 
finitely many forward solutions. The criterion is fulfilled if sufficiently many measurements are taken. Thus, the required number of measurements 
can be found by starting with a low number and incrementally increasing it until the criterion is fulfilled. The criterion also yields explicit error estimates for noisy data.

This work is closely related to \cite{harrach2021uniqueness} that gives an explicit construction of special measurements that uniquely determine the same number of unknowns in an inverse elliptic coefficient problem by a globally convergent Newton root-finding method. We also formulate our result for the same inverse Robin transmission problem as in \cite{harrach2021uniqueness} which is motivated by EIT-based corrosion detection and may be considered as a simpler variant of the Calder\'on problem. 
Our main advance in this work is the step from Newton root-finding to a convex semidefinite program. 
This allows utilizing a redundant set of given measurements, and eliminates the need of specially constructed
measurements. It also simplifies the underlying theory as it no longer requires simultaneously localized potentials,
and allows the criterion to be written using the Loewner order, which very naturally arises in elliptic inverse coefficient problems 
with finite resolution and finitely many measurements \cite{harrach2021introduction}. Also, to the knowledge of the author,
this work is the first connection between the emerging research fields of semidefinite optimization and inverse coefficient problems,
which might bring new inspiration to these important fields.

\section{Inverse problems for convex monotonous functions}\label{sect:inv_prob}

Let ,,$\leq$'' denote the entry-wise order on $\R^n$, and ,,$\preceq$'' denote the Loewner order on the space of symmetric
matrices $\sS_m\subseteq \R^{m\times m}$, with $n,m\in \N$. For $A\in \sS_m$ the largest eigenvalue is denoted by $\lambda_{\max}(A)$.

Given a-priori bounds $b>a>0$, we consider the inverse problem to
\begin{equation}\label{eq:inv_prob_F}
\text{determine } \quad x\in [a,b]^n\subset \R^n_+\quad \text{ from } \quad F(x)\in \sS_m\subseteq \R^{m\times m},
\end{equation}
where $F:\ \R^n_+\to \sS_m$ is assumed to be a continuously differentiable, convex and monotonically non-increasing matrix-valued function, i.e., for all $x,x^{(0)}\in \R^n_+$, and all $0\leq d\in \R^n$,
\begin{align}
\label{eq:F_mon} F'(x)d &\preceq 0,\\
\label{eq:F_conv} F(x)-F(x^{(0)}) &\succeq F'(x^{(0)})(x-x^{(0)}).
\end{align}
Such problems naturally arise in inverse coefficient problems in elliptic PDEs with finite resolution and finitely many measurements \cite{harrach2021introduction}.

Note that, here and in the following, we write the derivative of $F$ in a point $x\in \R^n_+$
as $F'(x)\in \LL(\R^n,\sS_m)$, so that $F'(x)d$ is a symmetric $m\times m$-matrix for all $d\in \R^n$.
Also note that a continuously differentiable function $F:\ \R^n_+\to \sS_m$ fulfills \eqref{eq:F_mon} and \eqref{eq:F_conv},
if and only if $F$ fulfills 
\begin{alignat*}{2}
F(x^{(0)})&\succeq F(x) && \quad \text{ for all } x,x^{(0)}\in \R^n_+,\ x^{(0)}\leq x,\\
F((1-t) x^{(0)}+t x)&\preceq  (1-t) F(x^{(0)})+t F(x) && \quad \text{ for all } x,x^{(0)}\in \R^n_+,\ t\in [0,1],
\end{alignat*}
cf., e.g., \cite[Lemma~2]{harrach2021introduction} for the only-if-part. The if part immediately follows from 
writing the directional derivative as differential quotient.

In this section we will derive a sufficient criterion for unique solvability of the finite-dimensional inverse problem \eqref{eq:inv_prob_F}
and for reformulating it as a convex optimization problem. Note that our criterion may appear technical at a first glance, but we stress that it only requires finitely many evaluations of directional derivatives of $F$, so that it can be easily checked in practice. Moreover, for the inverse Robin transmission problem considered in section \ref{sect:Robin}, 
we will show that the criterion will always be fulfilled if sufficiently many measurements are taken. 
Hence, the criterion allows to constructively determine the number of measurements that are required for a certain resolution and for convex reformulation by simply 
increasing the number of measurements until the criterion is fulfilled.

To formulate our result, let $e_j\in \R^n$ denote the $j$-th unit vector,
$\1\in \R^n$ denote the vector of ones, and $e_j':=\1-e_j$ is the vector containing zero in the $j$-th component and ones in all others.
For a matrix $A\in \R^{n\times n}$, $\norm{A}_2$ denotes the spectral norm, and for a number $\lambda\in \R$, $\lceil \lambda \rceil$ denotes
the ceiling function, i.e., the least integer greater than or equal to $\lambda$.

\begin{theorem}\label{thm:F_reformulation}
Let $F:\ \R^n_+\to \sS_m$, $n,m\geq 2$, be continuously differentiable, convex and monotonically non-increasing. 
If
\[
F'(z_{j,k})d_j\not\preceq 0\quad \text{ for all } k\in \{2,\ldots,K\},\ j\in \{1,\ldots,n\},
\]
where
\[
z_{j,k}:=\frac{a}{2}e_j' + \left(a+k\frac{a}{4(n-1)}\right)  e_j\in \R^n_+, \quad d_j:=\frac{2b-a}{a}(n-1)e_j' - \frac{1}{2}   e_j\in \R^n,
\]
and $K:=\lceil\frac{4(n-1)b}{a}\rceil-4n-3\in \N$, then the following holds
\begin{enumerate}[(a)]
\item $\hat x\in [a,b]^n$ is uniquely determined by knowledge of $\hat Y:=F(\hat x)$. $\hat x$ is the unique minimizer of the convex optimization problem
\begin{equation}\label{eq:sd_program_nonoise}
\verb|minimize |\ \norm{x}_1=\sum_{j=1}^n x_j \ \verb| subject to | \ x\in [a,b]^n,\ F(x)\preceq \hat Y.
\end{equation}
\item For $\hat x\in [a,b]^n$, $\hat Y:=F(\hat x)$, $\delta>0$, and $Y^\delta\in \sS_m$, with $\norm{\hat Y-Y^\delta}_2\leq \delta$,
 the convex optimization problem 
\begin{equation}\label{eq:sd_program_noisy}
\verb|minimize |\ \norm{x}_1=\sum_{j=1}^n x_j \ \verb| subject to | \ x\in [a,b]^n, \ F(x)\preceq Y^\delta+\delta I
\end{equation}
possesses a minimum, and every such minimum $x^\delta$ fulfills
\[
\norm{\hat x-x^\delta}_\infty\leq \frac{2\delta(n-1)}{\lambda}\quad \text{ with $\lambda:=\min_{\mycom{j=1,\ldots,n,}{k=2,\ldots,K}} \lambda_\text{max}(F'(z_{j,k})(d_j))>0$.}
\]
\end{enumerate}\end{theorem}

To prove Theorem~\ref{thm:F_reformulation} we will show the following lemmas.
\begin{lemma}\label{lemma:sum_crit_general}
Let $F:\ \R^n_+\to \sS_m$, $n,m\geq 2$, be continuously differentiable, convex and monotonically non-increasing.
If, for some $x\in \R^n_+$,
\begin{equation}\label{eq:sum_crit_general}
F'(x)((n-1)e_j'-e_j)\not\preceq 0 \quad \text{ for all } \quad  j\in \{1,\ldots,n\},
\end{equation}
then, for all $d\in \R^n$, and $y\in \R^n_+$, 
\begin{alignat}{2}
\label{eq:sum_crit_ass1quant}
\lambda_{\max}(F'(x)d)&< \frac{\lambda \norm{d}_\infty}{n-1} \quad && \text{ implies } \quad \sum_{j=1}^n d_j> 0, \quad \text{ and }\\
\lambda_{\max}(F(y)-F(x))&<  \frac{\lambda \norm{y-x}_\infty}{n-1}   \quad && \text{ implies } \quad \sum_{j=1}^n (y_j-x_j)> 0,
\label{eq:sum_crit_ass2quant}
\end{alignat}
with $\lambda:=\min_{j=1,\ldots,n} \lambda_\text{max}(F'(x)((n-1)e_j'-e_j))$.
\end{lemma}
\begin{proof}
We will show that, for all $d\in \R^n$,
\begin{equation}\label{eq:conv_mon_sum_aux}
\lambda_{\max}(F'(x)d)< \frac{\lambda \norm{d}_\infty}{n-1}
 \quad \text{ implies } \quad
\min_{j=1,\ldots,n} d_j > -\frac{1}{n-1} \max_{j=1,\ldots,n} d_j,
\end{equation}
which clearly implies \eqref{eq:sum_crit_ass1quant}. \eqref{eq:sum_crit_ass2quant} then follows from \eqref{eq:sum_crit_ass1quant} by 
the convexity property $F'(x)(y-x)\preceq F(y)-F(x)$.

We prove \eqref{eq:conv_mon_sum_aux} by contraposition and assume that there exists an index $k\in \{1,\ldots,n\}$ with 
\[
d_k=\min_{j=1,\ldots,n} d_j\leq -\frac{1}{n-1} \max_{j=1,\ldots,n} d_j.
\]
We have that either 
\[
\norm{d}_\infty=\max_{j=1,\ldots,n} d_j,\quad \text{ or } \quad \norm{d}_\infty=-\min_{j=1,\ldots,n} d_j=-d_k,
\]
and in both cases it follows that
\[
d_k\leq -\frac{1}{n-1}\norm{d}_\infty, \quad \text{ and thus } \quad d\leq -\frac{1}{n-1}\norm{d}_\infty e_k + \norm{d}_\infty e_k'.
\]
Hence, by \eqref{eq:sum_crit_general} and monotonicity,
\[
F'(x) d \succeq \frac{\norm{d}_\infty}{n-1} F'(x) \left( (n-1) e_k' - e_k \right),
\]
which yields that
\[
\lambda_{\max}(F'(x) d)\geq \frac{\norm{d}_\infty}{n-1} \lambda,
\]
so that \eqref{eq:conv_mon_sum_aux} is proven.\hfill $\Box$
\end{proof}

\begin{remark}\label{remark:conv_mon}
Lemma~\ref{lemma:sum_crit_general} can be considered a converse monotonicity result, as it yields that, for all $y\in \R^n_+$, with $y\neq x$,  
\begin{alignat*}{2}
F(y)&\preceq F(x)  \quad && \text{ implies } \quad \sum_{j=1}^n (y_j-x_j)> 0.
\end{alignat*}
\end{remark}

\begin{lemma}\label{lemma:finitely_many}
Let $F:\ \R^n_+\to \sS_m$, $n,m\geq 2$, be continuously differentiable, convex and monotonically non-increasing, and $b\geq a>0$.
If 
\[
\lambda:=\min_{\mycom{j=1,\ldots,n,}{k=2,\ldots,K}} \lambda_\text{max}(F'(z_{j,k})(d_j))>0,
\]
where $z_{j,k}\in \R_+^n$, $d_j\in \R^n$, and $K\in \N$ are defined as in Theorem~\ref{thm:F_reformulation}), then 
\[
\lambda_{\max}(F'(x)((n-1) e_j'-e_j))\geq \lambda \quad \text{ for all } x\in [a,b]^n,\ j\in \{1,\ldots,n\}.
\]
\end{lemma}
\begin{proof}
We will show that for all $j\in \{1,\ldots,n\}$ and $x\in [a,b]^n$, there exists $t\in [a+\frac{a}{2(n-1)},\ b+\frac{a}{2(n-1)}]\subset \R$, so that,
for all $0\leq \delta\leq \frac{a}{4(n-1)}$,
\begin{equation}\label{eq:finitely_many}
F'(x)((n-1) e_j'-e_j)\succeq F'\left( \frac{a}{2}e_j' + (t-\delta) e_j \right) d_j.
\end{equation}
Since $a+K\frac{a}{4(n-1)}
\geq b+\frac{a}{4(n-1)}$, we have that for every $t\in [a+\frac{a}{2(n-1)},\ b+\frac{a}{2(n-1)}]$,
there exists $k\in \{2,\ldots,K\}$, so that
\[
\delta:=t-\left(a+k\frac{a}{4(n-1)}\right)\leq \frac{a}{4(n-1)} \quad \text{ fulfills } \quad 
0\leq \delta\leq \frac{a}{4(n-1)}.
\]
Hence, if \eqref{eq:finitely_many} is proven, then
\[
F'(x)((n-1) e_j'-e_j)\succeq F'(z_{j,k}) d_j,
\]
so that the assertion follows.

To prove \eqref{eq:finitely_many}, let $j\in \{1,\ldots,n\}$, and $x\in [a,b]^n$. We define 
$
t:=x_j+\frac{a}{2(n-1)}.
$
Then, for all $0\leq \delta\leq \frac{a}{4(n-1)}$
\begin{align*}
(n-1)e_j'-e_j&=\frac{2(n-1)}{a}\left(\frac{a}{2} e_j' + (x_j-t) e_j\right)
\leq  \frac{2(n-1)}{a}\left( x- \frac{a}{2}e_j' - te_j \right)\\
&\leq  \frac{2(n-1)}{a}\left( x- \left( \frac{a}{2}e_j' + (t-\delta) e_j \right) \right),
\end{align*}
and
\begin{align*}
\lefteqn{\frac{2(n-1)}{a}\left( x- \left( \frac{a}{2}e_j' + (t-\delta) e_j \right) \right)}\\
& \leq \frac{2(n-1)}{a}\left( \left( b- \frac{a}{2}\right)e_j' + (x_j-t+\delta) e_j \right)\\
& = \frac{2b-a}{a}(n-1)e_j' + \frac{2(n-1)}{a} \left(- \frac{a}{2(n-1)} +\delta \right) e_j\\
& \leq \frac{2b-a}{a}(n-1)e_j' - \frac{1}{2}   e_j=d_j, 
\end{align*}
so that we obtain from monotonicity \eqref{eq:F_mon} and convexity \eqref{eq:F_conv}
\begin{align*}
\lefteqn{F'(x)((n-1)e_j'-e_j)}\\
&\succeq \frac{2(n-1)}{a} F'(x)\left( x- \left( \frac{a}{2}e_j' + (t-\delta) e_j \right) \right)\\
&\succeq \frac{2(n-1)}{a} \left(F(x)-F\left( \frac{a}{2}e_j' + (t-\delta) e_j \right)\right)\\
&\succeq \frac{2(n-1)}{a} F'\left( \frac{a}{2}e_j' + (t-\delta) e_j \right)\left( x- \left( \frac{a}{2}e_j' + (t-\delta) e_j \right) \right)\\
&\succeq F'\left( \frac{a}{2}e_j' + (t-\delta) e_j \right)d_j,
\end{align*}
which proves \eqref{eq:finitely_many} and thus the assertion.\hfill $\Box$
\end{proof}

\emph{Proof of Theorem~\ref{thm:F_reformulation}.}
Under the assumption of Theorem~\ref{thm:F_reformulation}, it follows from Lemma~\ref{lemma:finitely_many}, that the assumptions of Lemma~\ref{lemma:sum_crit_general}
are fulfilled, so that \eqref{eq:sum_crit_ass2quant} holds for all $x,y\in [a,b]^n$. In particular this yields that
$\hat Y:=F(\hat x)$ uniquely determines $\hat x\in [a,b]^n$.
Moreover, for every $x\in [a,b]^n$ with $x\neq \hat x$, and $F(x)\preceq \hat Y=F(\hat x)$, we obtain from Remark~\ref{remark:conv_mon} that
\[
\sum_{j=1}^n (x_j-\hat x_j)> 0,
\]
which shows that $\hat x$ is the unique minimizer of \eqref{eq:sd_program_nonoise}. This proves Theorem~\ref{thm:F_reformulation}(a).
To prove Theorem~\ref{thm:F_reformulation}(b), we note that the set of all $x\in [a,b]^n$ with $F(x)\preceq \hat Y^\delta+\delta I$ is compact and non-empty since it contains $\hat x$.
Hence, at least one minimizer of \eqref{eq:sd_program_noisy} exists. Every minimizer $x^\delta \in [a,b]^n$ fulfills 
\[
F(x^\delta)\preceq \hat Y^\delta+\delta I\preceq F(\hat x)+2\delta I.
\]
If $2\delta< \frac{\lambda\norm{x^\delta-\hat x}_\infty}{n-1}$, then \eqref{eq:sum_crit_ass2quant} 
would imply that
\[
\sum_{j=1}^n (x_j-\hat x_j)> 0,
\]
which contradicts the minimality of $x^\delta$. Hence $\norm{x^\delta-\hat x}_\infty\leq \frac{2\delta (n-1)}{\lambda}$.
\hfill $\Box$

\section{Application to an inverse elliptic coefficient problem}\label{sect:Robin}

We will now study the problem of determining a Robin transmission coefficient in an elliptic PDE from 
finitely many measurements. Using Theorem~\ref{thm:F_reformulation} we will show that 
this inverse coefficient problem can be rewritten as a uniquely solvable convex non-linear semidefinite optimization problem
if enough measurements are being used. This also gives a constructive criterion whether a certain number of measurements suffices
to determine the Robin parameter with a given desired resolution by convex optimization, and yields an error estimate for noisy data.

\subsection{The infinite-dimensional inverse Robin transmission problem}


Let $\Omega\subset \R^d$ ($d\geq 2$) be a bounded domain and $D\subset \Omega$ be an open subset with $\overline D\subset \Omega$.
$\Omega$ and $D$ are assumed to have Lipschitz boundaries, $\partial \Omega$ and $\Gamma:=\partial D$, and $\Omega\setminus D$ is assumed to be connected.
We consider the inverse problem of recovering the coefficient $\gamma\in L^\infty_+(\Gamma)$ in the elliptic Robin transmission problem 
\begin{alignat}{2}
\label{eq:Robin1} \Delta u_\gamma^g &=0 \quad && \text{ in } \Omega\setminus \Gamma,\\
\label{eq:Robin2} \partial_{\nu} u_\gamma^g|_{\partial \Omega}&= g\quad && \text{ on  }\partial \Omega,\\
\label{eq:Robin3} \llbracket u_\gamma^g \rrbracket_\Gamma&=  0 \quad && \text{ on }  \Gamma,\\
\label{eq:Robin4} \llbracket \partial_{\nu} u_\gamma^g\rrbracket_\Gamma &= \gamma u_\gamma^g \quad &&\text{ on }  \Gamma,
\end{alignat}
from the Neumann-Dirichlet-Operator 
\[
\Lambda(\gamma)g:=u_\gamma^g|_{\partial \Omega}, \quad \text{ where } \quad u_\gamma^g\in H^1(\Omega) \text{ solves \eqref{eq:Robin1}--\eqref{eq:Robin4}.}
\]
Using the Lax-Milgram theorem and the compactness of the trace operator from $H^{1}(\Omega)$ to $L^2(\partial \Omega)$, it
easily follows that \eqref{eq:Robin1}--\eqref{eq:Robin4} is uniquely solvable, and that $\Lambda(\gamma)\in \LL(L^2(\partial \Omega))$ is self-adjoint and compact.

We summarize and reformulate some known results on the Neumann-Dirichlet operator that motivate why the corresponding finite-dimensional inverse problem 
can be treated with the methods from section~\ref{sect:inv_prob}. In the following theorem ''$\leq$'' is to be understood pointwise almost everywhere
for $L^\infty$-functions, and ''$\succeq$'' is the Loewner order on the space of self-adjoint operators.


\begin{theorem}\label{thm:Robin_infdim}
$\Lambda:\ L^\infty_+(\Omega)\to \LL(L^2(\partial \Omega))$ is Fr\'echet differentiable. Moreover,
\begin{enumerate}[(a)]
\item $\Lambda$ is monotonically non-increasing and convex, i.e.,
\begin{alignat*}{2}
\Lambda'(\gamma)\delta&\preceq 0 \quad && \text{for all $\gamma\in L^\infty_+(\Omega)$, $0\leq \delta\in L^\infty(\Omega)$,}\\
\Lambda(\gamma)-\Lambda(\gamma^{(0)})&\succeq \Lambda'(\gamma^{(0)}) (\gamma-\gamma^{(0)}) \quad && \text{for all $\gamma,\gamma^{(0)}\in L^\infty_+(\Omega)$.}
\end{alignat*}
\item For all $\gamma\in L^\infty_+(\Omega)$, $C>0$, and
$M\subseteq \Gamma$ measurable with positive measure
\[
\Lambda'(\gamma)(C\chi_{\Gamma\setminus M}-\chi_M)\not\preceq 0.
\]
\item For all $\gamma_1,\gamma_2\in L^\infty_+(\partial \Omega)$
\[
\gamma_1\leq \gamma_2 \quad \text{ if and only if } \quad \Lambda(\gamma_1)\succeq \Lambda(\gamma_2).
\]
In particular, $\Lambda(\gamma)$ uniquely determines $\gamma$. 
\end{enumerate}
\end{theorem}
\begin{proof}
Fr\'echet differentiability, monotonicity and convexity of $\Lambda$ are shown in \cite[Lemma~5]{harrach2021uniqueness}, cf.\ also
\cite[Lemma 4.1]{harrach2019global}. (b) follows from the localized potentials result in \cite[Lemma 4.3]{harrach2019global}.
The ''only if''-part in (c) follows from (a), and the ''if''-part in (c) easily follows from using (b) together with the convexity inequality in (a).
\hfill $\Box$
\end{proof}

\subsection{The inverse problem with finitely many measurements}

We now consider the inverse Robin transmission problem with finite resolution and finitely many measurements as in \cite{harrach2021uniqueness}. 
We assume that the unknown coefficient function $\gamma\in L^\infty_+(\Gamma)$ is piecewise constant on an a-priori known partition of $\Gamma$,
i.e.
\[
\gamma(x)=\sum_{j=1}^n \gamma_j \chi_{\Gamma_j}(x), \quad \text{ with } \Gamma=\bigcup_{j=1}^n \Gamma_j,
\]
where $\Gamma_1,\ldots,\Gamma_n$, $n\geq 2$, are pairwise disjoint measurable subsets of $\Gamma$. For the ease of notation, we identify a piecewise constant function $\gamma\in L^\infty(\Gamma)$
with the vector $\gamma=(\gamma_1,\ldots,\gamma_n)^T\in \R^n$ in the following. We also assume that we know a-priori bounds $b>a>0$, so that $\gamma\in [a,b]^n$.

We aim to reconstruct $\gamma\in [a,b]^n$ from finitely many measurements of $\Lambda(\gamma)$. More precisely, we assume that $(g_j)_{j\in \N}\subseteq L^2(\partial \Omega)$ has dense span in $L^2(\partial \Omega)$,
and that we can measure
\[
F(\gamma):=\left(\int_{\partial \Omega} g_j \Lambda(\gamma) g_k \dx[s]\right)_{j,k=1,\ldots,m}\in \R^{m\times m}
\]
for some number $m\in \N$. Thus, the question whether a certain number of measurements determine the unknown coefficient with a certain resolution can be written as
the problem to
\begin{equation}\label{eq:inv_problem_Robin}
\text{ determine } \quad \gamma\in [a,b]^n \quad \text{ from } \quad F(\gamma)\in \R^{m\times m}.
\end{equation}
Using our results in section~\ref{sect:inv_prob} we can now show that this inverse problem is uniquely solvable if sufficiently many measurements are being used, and that it can be equivalently reformulated 
as a convex semidefinite program. 

\begin{theorem}
\begin{enumerate}[(a)]
\item If $m\in \N$ is sufficiently large then $\hat Y:=F(\hat \gamma)\in  \sS_m$ uniquely determines $\hat \gamma\in [a,b]^n$.
$\hat \gamma$ is the unique minimizer of the convex semi-definite optimization problem
\[
\verb|minimize |\ \norm{\gamma}_1=\sum_{j=1}^n \gamma_j \ \verb| subject to | \ \gamma\in [a,b]^n,\ F(\gamma)\preceq \hat Y.
\]
\item The assertion in (a) holds if all matrices $F'(z_{j,k})d_j\in \sS_m$,
(with $z_{j,k}\in \R_+^n$, $d_j\in \R^n$, and $K\in \N$ given in Theorem~\ref{thm:F_reformulation})
possess at least one positive eigenvalue. This criterion is fulfilled for sufficiently large $m\in \N$.
Moreover, for $\delta>0$, and $Y^\delta\in \sS_m$, with $\norm{\hat Y-Y^\delta}_2\leq \delta$, the convex optimization problem 
\begin{equation*}
\verb|minimize |\ \norm{\gamma}_1=\sum_{j=1}^n \gamma_j \ \verb| subject to | \ \gamma\in [a,b]^n,\ F(\gamma)\preceq Y^\delta+\delta I
\end{equation*}
possesses a minimum, and every such minimum $\gamma^\delta$ fulfills
\[
\norm{\hat \gamma-\gamma^\delta}_\infty\leq \frac{2\delta(n-1)}{\lambda}\quad \text{ with $\lambda:=\min_{\mycom{j=1,\ldots,n,}{k=2,\ldots,K}} \lambda_\text{max}(F'(z_{j,k})(d_j))>0$.}
\]
\end{enumerate}
\end{theorem}
\begin{proof}
$F(\gamma)$ is a symmetric matrix since $\Lambda(\gamma)$ is self-adjoint. Fr\'echet differentiability, monotonicity and convexity of $F:\ \R^n_+\to \sS_m$ immediately 
follow from the corresponding properties of $\Lambda$ in Theorem~\ref{thm:Robin_infdim}. 
For all $j=1,\ldots,n$, $k=2,\ldots,K$, we have that $\Lambda'(z_{j,k})d_j\not\preceq 0$ by  Theorem~\ref{thm:Robin_infdim}(b).
By density, it follows that for all $j=1,\ldots,n$, $k=2,\ldots,K$ there exists $m\in \N$ so that $F'(z_{j,k})d_j\not\preceq 0$, and since
there are only finitely many such combinations of $j$ and $k$, there exists $m\in N$, so that all these matrices possess a positive eigenvalue. Hence, the assertions
follow from Theorem~\ref{thm:F_reformulation}.\hfill $\Box$
\end{proof}


%
\section*{Conflict of interest and data availability statement}
The author declares that he has no conflict of interest. Data sharing not applicable to this article as no datasets were generated or analysed during the current study


\bibliographystyle{spmpsci}
\bibliography{literaturliste}

\end{document}